\documentclass[12pt,twoside]{article}
\usepackage[left=1.5cm,right=1.5cm,top=3cm,bottom=2cm]{geometry}
\usepackage{paralist}
\usepackage{marginnote}
\usepackage{amsfonts}
\usepackage{enumitem}
\usepackage{tikz}
\usetikzlibrary{matrix}
\usepackage{amsthm}
\usepackage{amsrefs}
\usepackage{amssymb}
\usepackage{amsmath}
\usepackage{graphicx}
\usepackage{varwidth}
\usepackage{hyperref}
\usepackage{fancyhdr}
\usepackage{stmaryrd}
\usepackage{multirow}

\usepackage{blkarray}
\usepackage[autostyle]{csquotes}
\usepackage[mathscr]{euscript}
\hypersetup{
    colorlinks=true,
    linkcolor=black,
    filecolor=magenta,      
    urlcolor=black,
    citecolor=black
}
\DeclareMathAlphabet{\mathcalligra}{T1}{calligra}{c}{h}
\providecommand{\U}[1]{\protect\rule{.1in}{.1in}}
\newtheorem{theorem}{Theorem}[section]
\newtheorem{proposition}[theorem]{Proposition}
\newtheorem{lemma}[theorem]{Lemma}

\newtheorem{corollary}[theorem]{Corollary}

\newtheorem{remark}[theorem]{Remark}
\let\oldremark\remark
\renewcommand{\remark}{\oldremark\normalfont}

\newtheorem{example}[theorem]{Example}
\let\oldexample\example
\renewcommand{\example}{\oldexample\normalfont}

\newtheorem{examples}[theorem]{Examples}
\let\oldexamples\examples
\renewcommand{\examples}{\oldexamples\normalfont}

\def\<{{\langle}}
\def\>{{\rangle}}

\def\bea{\begin{eqnarray*} }
\def\eea{\end{eqnarray*} }
\def\be{\begin{equation} }
\def\ee{\end{equation} }

\def\qed{\ifhmode\unskip\nobreak\fi\ifmmode\ifinner
\else\hskip5 pt \fi\fi\hbox{\hskip5 pt \vrule width4 pt
height6 pt  depth1.5 pt \hskip 1pt }}

\DeclareMathOperator*{\diver}{div}
\DeclareMathOperator*{\Vol}{Vol}

\DeclareMathOperator*{\supp}{supp}

\DeclareMathOperator*{\expo}{exp}

\DeclareMathOperator*{\grad}{grad}
\DeclareMathOperator*{\ess}{ess}
\DeclareMathOperator*{\Lip}{Lip}
\DeclareMathOperator*{\av}{av}
\begin{document}

\title{Spectral estimates and discreteness of spectra under Riemannian submersions}
\author{Panagiotis Polymerakis}
\date{}

\maketitle

\renewcommand{\thefootnote}{\fnsymbol{footnote}}
\footnotetext{\emph{Date:} \today} 
\renewcommand{\thefootnote}{\arabic{footnote}}

\renewcommand{\thefootnote}{\fnsymbol{footnote}}
\renewcommand{\thefootnote}{\arabic{footnote}}

\renewcommand{\thefootnote}{\fnsymbol{footnote}}
\renewcommand{\thefootnote}{\arabic{footnote}}

\begin{abstract}
For Riemannian submersions, we establish some estimates for the spectrum of the total space in terms of the spectrum of the base space and the geometry of the fibers. In particular, for Riemannian submersions of complete manifolds with closed fibers of bounded mean curvature, we show that the spectrum of the base space is discrete if and only if the spectrum of the total space is discrete.
\end{abstract}

\section{Introduction}

The spectrum of the Laplacian on a Riemannian manifold is an isometric invariant whose relation with the geometry of the manifold is not comprehended completely. In particular, its behavior under maps between Riemannian manifolds, which respect the geometry of the manifolds to some extent, remains largely unclear. In this paper, we study the behavior of the spectrum under Riemannian submersions.

The notion of Riemannian submersion was introduced in the sixties as a tool to study the geometry of a manifold in terms of the geometry of simpler components, namely, the base space and the fibers. Similarly to other geometric quantities, it is natural to describe the spectrum of the total space in terms of the geometry and the spectrum of the base space and the fibers. Of course, the term geometry of the fibers refers both to the intrinsic and the extrinsic geometry of the fibers as submanifolds of the total space. There are various results on the spectrum of closed total spaces, in case the submersion has totally geodesic, or minimal fibers, or fibers of basic mean curvature (cf. for instance \cite{MR2963622} and the references therein).
Our results focus mostly on the non-compact case, which is in general more complicated and less understood.

To set the stage, let $p \colon M_{2} \to M_{1}$ be a Riemannian submersion and denote by $F_{x} := p^{-1}(x)$ the fiber over $x \in M_{1}$. In the first part of the paper, we establish a lower bound for the bottom of the spectrum of the total space, under the assumption that the (unnormalized) mean curvature $H$ of the fibers is bounded in a specific way. In particular, we extend the recent result of \cite{MR3787357} about Riemannian submersions. According to \cite[Theorem 1.1]{MR3787357}, if $M_{1}$ is the $m$-dimensional hyperbolic space $\mathbb{H}^{m}$, and the mean curvature vector field of the fibers is bounded by $\| H \| \leq C \leq m-1$, then the bottom of the spectrum of the Laplacian on $M_{2}$ satisfies
\[
\lambda_{0}(M_{2}) \geq \frac{(m -1 - C)^{2}}{4}.
\]
It should be noticed that $m-1 = 2 \sqrt{\lambda_{0}(\mathbb{H}^{m})}$. This result is extended in \cite[Theorem 5.1]{MR3787357} to the case where the base manifold is Hadamard with sectional curvature bounded from above by a negative constant, or the base manifold is a warped product of some special form. In its general version, in the assumption and the conclusion of the above formulation, $m-1$ is replaced by a certain lower bound for $2 \sqrt{\lambda_{0}(M_{1})}$. Our first result generalizes this estimate in various directions, and provides some information in the case where the equality holds and $\lambda_{0}(M_{1})$ is an isolated point of the spectrum of the Laplacian on $M_{1}$.

\begin{theorem}\label{lower bound}
Let $p \colon M_{2} \to M_{1}$ be a Riemannian submersion, such that the mean curvature of the fibers satisfies $\| H \| \leq C \leq 2 \sqrt{\lambda_{0}(M_{1})}$. Then
\[
\lambda_{0}(M_{2}) \geq (\sqrt{\lambda_{0}(M_{1})} - C/2)^{2} + \inf_{x \in M_{1}} \lambda_{0}(F_{x}).
\]
If, in addition, the equality holds and $\lambda_{0}(M_{1}) \notin \sigma_{\ess}(M_{1})$, then $\lambda_{0}(F_{x})$ is almost everywhere equal to its infimum.
\end{theorem}

It should be emphasized that, in this theorem, there are no assumptions on the geometry or the topology of the base space. In particular, Theorem \ref{lower bound} gives a quite natural (and sharper than \cite[Theorem 5.1]{MR3787357}) estimate for submersions over negatively curved symmetric spaces, and yields an analogous lower bound if the base manifold is a complete, negatively curved, locally symmetric space. Moreover, the manifolds involved in Theorem \ref{lower bound} do not have to be complete, which in the sequel allows us to derive a similar estimate involving the bottoms of the essential spectra, by exploiting the Decomposition Principle.

Conceptually, it seems interesting that the last term in the estimate of Theorem \ref{lower bound} shows up, while in \cite{MR3787357} the intrinsic geometry of the fibers does not play any role. For example, equality in the estimate of Theorem \ref{lower bound} holds if $M_{2}$ is the Riemannian product $M_{1} \times F$ for any Riemannian manifold $F$. For the aforementioned reason, equality in the estimate of \cite[Theorem 1]{MR3787357} holds for $M_{2} = \mathbb{H}^{k} \times F$ only for Riemannian manifolds $F$ with $\lambda_{0}(F) = 0$.

In general, it is quite important to establish lower bounds for the bottom of the spectrum, or even deduce whether it is zero or not. It follows from Theorem \ref{lower bound} that if $p \colon M_{2} \to M_{1}$ is a Riemannian submersion with minimal fibers, $M_{1}$ is closed, and $\lambda_{0}(M_{2}) = 0$, then the bottom of the spectrum of almost any fiber is zero, since the spectrum of the Laplacian on $M_{1}$ is discrete. As Example \ref{discontinuous} shows, in this case, the bottom of the spectrum of some fibers may be positive. In principle, in order to deduce that the bottom of the spectrum is positive, one needs information on the global geometry of the underlying manifold. However, in the above setting, we obtain that $\lambda_{0}(M_{2}) > 0$ if $\lambda_{0}(F_{x}) > 0$ for any $x$ in a set of positive measure.

In the second part of the paper, we consider Riemannian submersions $p \colon M_{2} \to M_{1}$ of complete manifolds with closed fibers. Such submersions under further constraints, are studied in \cite{MR2891739}. According to \cite[Theorem 1]{MR2891739}, if the fibers are minimal, then the spectra and the essential spectra of the Laplacians satisfy $\sigma(M_{1}) \subset \sigma(M_{2})$ and $\sigma_{\ess}(M_{1}) \subset \sigma_{\ess}(M_{2})$.
In this paper, we extend this result in a natural way. Instead of comparing the spectra of the Laplacians, we compare the spectrum of a Schr\"{o}dinger operator on $M_{1}$, whose potential is determined by the volume of the fibers, with the spectrum of the Laplacian on $M_{2}$. To be more precise, let $V(x)$ be the volume of the fiber over $x \in M_{1}$, and consider the Schr\"{o}dinger operator
\[
S := \Delta - \frac{\Delta \sqrt{V}}{\sqrt{V}}
\]
on $M_{1}$. The potential of this operator measures the deviation of $\sqrt{V}$ from being harmonic. In particular, if the submersion has minimal fibers (or more generally, fibers of constant volume), then $S$ coincides with the Laplacian on $M_{1}$. 

In the case where $M_{2}$ is closed and the submersion has fibers of basic mean curvature, Bordoni \cite{Bordoni} considered the restriction $\Delta_{c}$ of the Laplacian acting on lifted functions and the restriction $\Delta_{0}$ of the Laplacian acting on functions whose average is zero on any fiber. In \cite[Theorem 1.6]{Bordoni}, he showed that $\sigma(M_{2}) = \sigma( \Delta_{c}) \cup \sigma (\Delta_{0})$. In this situation, the spectrum of $S$ coincides with the spectrum of $\Delta_{c}$. It should be noticed that expressing the latter one as the spectrum of a Schr\"{o}dinger operator on the base manifold allows us to relate it more easily to the spectrum of the Laplacian on the base manifold.

For submersions of complete manifolds with closed fibers, we compare the bottoms of the (essential) spectra of $S$ and of the Laplacian on the total space. If the submersion has fibers of basic mean curvature, we prove that the (essential) spectrum of $S$ is contained in the (essential, respectively) spectrum of the Laplacian on $M_{2}$. This is formulated in the following generalization of \cite[Theorem 1]{MR2891739}.

\begin{theorem}\label{thm2}
Let $p \colon M_{2} \to  M_{1}$ be a Riemannian submersion of complete manifolds, with closed fibers. Then $\lambda_{0}(M_{2}) \leq \lambda_{0}(S)$ and $\lambda_{0}^{\ess}(M_{2}) \leq \lambda_{0}^{\ess}(S)$. If, in addition, the fibers have basic mean curvature, then $\sigma(S) \subset \sigma(M_{2})$ and $\sigma_{\ess}(S) \subset \sigma_{\ess}(M_{2})$.
\end{theorem}

Finally, we consider the problem of discreteness of spectra under Riemannian submersions. A Riemannian manifold $M$ has \textit{discrete spectrum} if the essential spectrum of the Laplacian on $M$ is empty. Although there are many results establishing connections between the geometry of $M$ and the discreteness of the spectrum of $M$ (cf. for example \cite{MR2891739} and the references provided there), their relation is not comprehended completely. 

In our context, there are examples of Riemannian submersions $p \colon M_{2} \to M_{1}$ of complete manifolds with closed fibers, such that $M_{1}$ has discrete spectrum and $M_{2}$ does not have discrete spectrum, or $M_{1}$ does not have discrete spectrum and $M_{2}$ has discrete spectrum (cf. \cite[Subsection 4.2]{MR2891739}). In \cite{MR2891739}, it is proved that if $p$ has minimal fibers, then $M_{1}$ has discrete spectrum if and only if $M_{2}$ has discrete spectrum. As an application of Theorems \ref{lower bound} and \ref{thm2}, we extend this equivalence under the weaker assumption that the fibers have bounded mean curvature.

\begin{corollary}\label{discrete spectrum}
Let $p \colon M_{2} \to M_{1}$ be a Riemannian submersion of complete manifolds, with closed fibers of bounded mean curvature. Then $M_{1}$ has discrete spectrum if and only if $M_{2}$ has discrete spectrum.
\end{corollary}

The paper is organized as follows: In Section \ref{Section Preliminaries}, we give some preliminaries involving the spectrum of Schr\"{o}dinger operators, and recall some basic facts on Riemannian submersions. In Section \ref{Section Lower bound}, we study Riemannian submersions with fibers of bounded mean curvature and establish Theorem \ref{lower bound}. In Section \ref{Section closed fibers}, we consider Riemannian submersions with closed fibers and prove Theorem \ref{thm2} and Corollary \ref{discrete spectrum}.

\medskip

\textbf{Acknowledgements.} I would like to thank Werner Ballmann and Dorothee Sch\"{u}th for their helpful comments and remarks. I am also grateful to the Max Planck Institute for Mathematics in Bonn for its support and hospitality.

\section{Preliminaries}\label{Section Preliminaries}

Throughout this paper manifolds are assumed to be connected and without boundary, unless otherwise stated. For a possibly non-connected Riemannian manifold $M$, we denote by $\Delta$ the non-negative definite Laplacian on $M$. A \textit{Schr\"{o}dinger operator} on $M$ is an operator of the form $S = \Delta + V$, with $V \in C^{\infty}(M)$, such that
\begin{equation}\label{bounded from below}
\langle Sf ,f \rangle_{L^{2}(M)} \geq c \| f \|^{2}_{L^{2}(M)}
\end{equation}
for some $c \in \mathbb{R}$ and any $f \in C^{\infty}_{c}(M)$. Then the operator
\begin{equation}\label{operator}
S \colon C^{\infty}_{c}(M) \subset L^{2}(M) \to L^{2}(M)
\end{equation}
is densely defined, symmetric and bounded from below. Therefore, it admits Friedrichs extension. We denote the spectrum and the essential spectrum of its Friedrichs extension by $\sigma(S)$ and $\sigma_{\ess}(S)$, respectively, and their \textit{bottoms} (that is, their minimums) by $\lambda_{0}(S)$ and $\lambda_{0}^{\ess}(S)$, respectively. In case of the Laplacian (that is, $V = 0$) these sets and quantities are denoted by $\sigma(M)$, $\sigma_{\ess}(M)$ and $\lambda_{0}(M)$, $\lambda_{0}^{\ess}(M)$, respectively. The spectrum of $S$ is called \textit{discrete} if $\sigma_{\ess}(S)$ is empty. In this case, we have by definition that $\lambda_{0}^{\ess}(S) = + \infty$.

For a non-zero, compactly supported, Lipschitz function $f$ on $M$, the \textit{Rayleigh quotient} of $f$ with respect to $S$ is defined by
\[
\mathcal{R}_{S}(f) := \frac{\int_{M} (\| \grad f \|^{2} + Vf^{2})}{\int_{M} f^{2}}.
\] 
The Rayleigh quotient of $f$ with respect to the Laplacian is denoted by $\mathcal{R}(f)$. The next characterization for the bottom of the spectrum of a Schr\"{o}dinger operator follows from Rayleigh's Theorem and the fact that the Friedrichs extension of an operator preserves its lower bound (cf. for instance \cite[Subsection 2.1]{Mine 2} and the references therein).
\begin{proposition}\label{bottom}
Let $S$ be a Schr\"{o}dinger operator on a Riemannian manifold $M$. Then the bottom of the spectrum of $S$ is given by
\[
\lambda_{0}(S) = \inf_{f} \mathcal{R}_{S}(f),
\]
where the infimum is taken over all $f \in C^{\infty}_{c}(M) \smallsetminus \{0\}$, or over all $f \in \Lip_{c}(M) \smallsetminus \{0\}$.
\end{proposition}

\begin{proposition}\label{disjoint}
Let $S$ be a Schr\"{o}dinger operator on a Riemannian manifold $M$. Then for any sequence $(f_{n})_{n \in \mathbb{N}} \subset C^{\infty}_{c}(M) \smallsetminus \{0\}$, with $\supp f_{n}$ pairwise disjoint, we have that
\[
\lambda_{0}^{\ess}(S) \leq \liminf_{n} \mathcal{R}_{S}(f_{n}).
\]
\end{proposition}

\begin{proof}
If the right hand side is infinite, there is nothing to prove. If it is finite, we denote it by $\lambda$, and after passing to a subsequence, if necessary, we may assume that $\mathcal{R}_{S}(f_{n}) \rightarrow \lambda$.
For any $\varepsilon > 0$, there exists $n_{0} \in \mathbb{N}$ such that $\mathcal{R}_{S}(f_{n}) < \lambda + \varepsilon$ for any $n \geq n_{0}$. Consider the infinite dimensional space $\mathcal{H}_{\varepsilon}$ spanned by $\{ f_{n} : n \geq n_{0} \}$. Any element $g \in \mathcal{H}_{\varepsilon} \smallsetminus \{0\}$ is of the form $g = \sum_{n=n_{0}}^{n_{0} + k} a_{n} f_{n}$ for some $k \in \mathbb{N}$ and $a_{n} \in \mathbb{R}$, $n_{0} \leq n \leq n_{0}+k$. The assumption that the supports of $f_{n}$ are pairwise disjoint yields that
\[
\mathcal{R}_{S}(g) = \frac{\sum_{n=n_{0}}^{n_{0} + k} a_{n}^{2} \int_{M} (\|\grad f_{n}\|^{2} + Vf_{n}^{2})}{\sum_{n=n_{0}}^{n_{0} + k} a_{n}^{2} \int_{M} f_{n}^{2}} \leq \max_{n_{0} \leq n \leq n_{0} + k} \mathcal{R}_{S}(f_{n}) < \lambda + \varepsilon.
\]
Since $\varepsilon > 0$ is arbitrary, we conclude from \cite[Proposition 2.1]{MR592568} that $\lambda_{0}^{\ess}(S) \leq \lambda$. \qed
\end{proof}\medskip

Let $\varphi$ be a positive, smooth function on $M$ such that $S \varphi = \lambda \varphi$ for some $\lambda \in \mathbb{R}$. Denote by $L^{2}_{\varphi}(M)$ the $L^{2}$-space of $M$ with respect to the measure $\varphi^{2} \text{d} \Vol$, where $\text{d} \Vol$ stands for the volume element of $M$ induced from its Riemannian metric. It is immediate to verify that the map $\mu_{\varphi} \colon L^{2}_{\varphi}(M) \to L^{2}(M)$, given by $\mu_{\varphi}(u) = \varphi u$, is an isometric isomorphism. The \textit{renormalization} of $S$ with respect to $\varphi$ is defined by
\[
S_{\varphi} := \mu_{\varphi}^{-1} \circ (S^{(F)} - \lambda) \circ \mu_{\varphi}, \text{ with } \mathcal{D}(S_{\varphi}) := \mu_{\varphi}^{-1}(\mathcal{D}(S^{(F)})),
\]
where $S^{(F)}$ is the Friedrichs extension of $S$ considered as in (\ref{operator}), and $\mathcal{D}(\cdot)$ denotes the domain of the operator. More details on the renormalization of Schr\"{o}dinger operators may be found in \cite[Section 7]{Mine}. Given $f \in C^{\infty}_{c}(M)$, it is straightforward to compute
\begin{equation}\label{renormalized exp}
S_{\varphi}f = \Delta f - \frac{2}{\varphi} \langle \grad \varphi , \grad f \rangle,
\end{equation}
which shows that $S_{\varphi}$ is a weighted Laplacian on $M$. The Rayleigh quotient of a non-zero $f \in C^{\infty}_{c}(M)$ with respect to $S_{\varphi}$ is given by
\[
\mathcal{R}_{S_{\varphi}}(f) := \frac{\langle S_{\varphi} f , f \rangle_{L^{2}_{\varphi}(M)}}{\| f  \|_{L^{2}_{\varphi}(M)}^{2}} = \frac{\int_{M} \| \grad f \|^{2} \varphi^{2}}{\int_{M} f^{2}\varphi^{2}}.
\]

\begin{lemma}\label{renormalized}
For any $f \in C^{\infty}_{c}(M) \smallsetminus \{0\}$ and $C \in \mathbb{R}$, we have that:
\begin{enumerate}[topsep=0pt,itemsep=-1pt,partopsep=1ex,parsep=0.5ex,leftmargin=*, label=(\roman*), align=left, labelsep=0em]
\item $\mathcal{R}_{S_{\varphi}}(f) = \mathcal{R}_{S}(\varphi f) - \lambda,$
\item $ \| (S_{\varphi} - C ) f \|_{L^{2}_{\varphi}(M)} =  \| (S - \lambda - C ) (\varphi f) \|_{L^{2}(M)}$.
\end{enumerate}
\end{lemma}

\begin{proof}
Both statements follow easily from the definition of $S_{\varphi}$ and the fact that $\mu_{\varphi}$ is an isometric isomorphism. \qed
\end{proof}\medskip

We now consider Schr\"{o}dinger operators on complete Riemannian manifolds. According to the next proposition, a Schr\"{o}dinger operator on a complete Riemannian manifold is essentially self-adjoint; that is, the Friedrichs extension of $S$ coincides with the closure of $S$ considered as in (\ref{operator}). This allows us to characterize the spectrum of the operator in terms of compactly supported smooth functions.

\begin{proposition}
Let $S$ be a Schr\"{o}dinger operator on a complete Riemannian manifold $M$. Then $S$ is essentially self-adjoint.
\end{proposition}

\begin{proof}
By virtue of \cite[Theorem 1]{MR562550} and (\ref{bounded from below}), we have that there exists $\lambda \in \mathbb{R}$ and a positive $\varphi \in C^{\infty}(M)$ such that $S \varphi = \lambda \varphi$. Denote by $S_{\varphi}$ the renormalization of $S$ with respect to $\varphi$. Then \cite[Theorem 2.2]{MR2218016} implies that the operator
\[
S_{\varphi} \colon C^{\infty}_{c}(M) \subset L^{2}_{\varphi}(M) \to L^{2}_{\varphi}(M)
\]
is essentially self-adjoint, $M$ being complete. Taking into account that this operator corresponds to $S - \lambda$ (considered as in (\ref{operator})) under the isometric isomorphism $\mu_{\varphi}$, we derive that $S$ is essentially self-adjoint. \qed
\end{proof}\medskip

Let $S$ be a Schr\"{o}dinger operator on a complete Riemannian manifold $M$. For $\lambda \in \mathbb{R}$, a sequence $(f_{n})_{n \in \mathbb{N}} \subset C^{\infty}_{c}(M) \smallsetminus \{0\}$ is called \textit{characteristic sequence} for $S$ and $\lambda$, if
\[
\frac{\| (S - \lambda)f_{n} \|_{L^{2}(M)}}{\| f_{n} \|_{L^{2}(M)}} \rightarrow 0, \text{ as } n\rightarrow + \infty.
\]
The next propositions follow from the Decomposition Principle \cite{MR544241} and the fact that the spectrum of a self-adjoint operator consists of approximate eigenvalues of the operator.

\begin{proposition}\label{spectrum}
Let $S$ be a Schr\"{o}dinger operator on a complete Riemannian manifold $M$, and consider $\lambda \in \mathbb{R}$. Then:
\begin{enumerate}[topsep=0pt,itemsep=-1pt,partopsep=1ex,parsep=0.5ex,leftmargin=*, label=(\roman*), align=left, labelsep=0em]
\item $\lambda \in \sigma(S)$ if and only there exists a characteristic sequence for $S$ and $\lambda$,
\item $\lambda \in \sigma_{\ess}(S)$ if and only if there exists a characteristic sequence $(f_{n})_{n \in \mathbb{N}}$ for $S$ and $\lambda$, with $\supp f_{n}$ pairwise disjoint.
\end{enumerate}
\end{proposition}

\begin{proposition}[{\cite[Proposition 3.2]{MR2891739}}]\label{exhausting}
Let $S$ be a Schr\"{o}dinger operator on a complete Riemannian manifold $M$, and $(K_{n})_{n \in \mathbb{N}}$ an exhausting sequence of $M$ consisting of compact subsets of $M$. Then the bottom of the essential spectrum of $S$ is given by
\[
\lambda_{0}^{\ess}(S) = \lim_{n} \lambda_{0}(S, M \smallsetminus K_{n}),
\]
where $\lambda_{0}(S, M \smallsetminus K_{n})$ is the bottom of the spectrum of $S$ on $M \smallsetminus K_{n}$. In particular, the spectrum of $S$ is discrete if and only if the right hand side limit is infinite.
\end{proposition}

The following property of the bottom of the essential spectrum is an immediate consequence of Propositions \ref{bottom} and \ref{exhausting}.

\begin{corollary}\label{bottom of essential spectrum}
Let $S$ be a Schr\"{o}dinger operator on a complete Riemannian manifold $M$. Then there exists $(f_{n})_{n \in \mathbb{N}} \subset C^{\infty}_{c}(M) \smallsetminus \{0\}$, with $\supp f_{n}$ pairwise disjoint, such that $\mathcal{R}_{S}(f_{n}) \rightarrow \lambda_{0}^{\ess}(S)$.
\end{corollary}

\subsection{Riemannian submersions}

Let $M_{1}$, $M_{2}$ be Riemannian manifolds with $\dim (M_{2}) > \dim (M_{1})$. A surjective, smooth map $p \colon M_{2} \to M_{1}$ is called a \textit{submersion} if its differential is surjective at any point $y \in M_{2}$. For any $x \in M_{1}$, the \textit{fiber} $F_{x} := p^{-1}(x)$ over $x$ is a possibly non-connected submanifold of $M_{2}$. The kernel of $p_{*y}$ is called the \textit{vertical space} at $y$, and is denoted by $(T_{y}M_{2})^{v}$. Evidently, the vertical space at $y$ is the tangent space of the fiber $F_{p(y)}$. The \textit{horizontal space} $(T_{y}M_{2})^{h}$ at $y$ is defined as the orthogonal complement of the vertical space at $y$. The submersion $p$ is called \textit{Riemannian submersion} if the restriction $p_{*y} \colon (T_{y}M_{2})^{h} \to T_{p(y)}M_{1}$ is an isometry for any $y \in M_{2}$. For more details on Riemannian submersions, see for example \cite{MR2110043}.

Let $p \colon M_{2} \to M_{1}$ be a Riemannian submersion. A vector field $Y$ on $M_{2}$ is called \textit{horizontal} (\textit{vertical}) if $Y(y)$ belongs to the horizontal (vertical, respectively) space at $y$ for any $y \in M_{2}$. It is clear that any vector field $Y$ on $M_{2}$ is written uniquely as $Y = Y^{h} + Y^{v}$, with $Y^{h}$ horizontal and $Y^{v}$ vertical. Any vector field $X$ on $M_{1}$ has a unique horizontal lift on $M_{2}$, which is denoted by $\tilde{X}$; that is, $\tilde{X}$ is horizontal and $p_{*}\tilde{X} = X$. A vector field $Y$ on $M_{2}$ is called \textit{basic} if $Y = \tilde{X}$ for some vector field $X$ on $M_{1}$.

We denote by $H$ the (unnormalized) mean curvature of the fibers, which is defined by
\[
H(y) := \sum_{i=1}^{k} \alpha(e_{i},e_{i}),
\]
where $\alpha(\cdot , \cdot)$ is the second fundamental form of $F_{p(y)}$, and $\{e_{i}\}_{i=1}^{k}$ is an orthonormal basis of $(T_{y}M_{2})^{v}$. The Riemannian submersion $p$ has \textit{minimal fibers}, \textit{fibers of basic mean curvature}, or \textit{fibers of bounded mean curvature} if $H=0$, $H$ is basic, or $\| H \|$ is bounded, respectively.

Given a function $f \colon M_{1} \to \mathbb{R}$, the function $\tilde{f} := f \circ p$ is called the \textit{lift} of $f$ on $M_{2}$. The next lemma provides a simple expression for the gradient and the Laplacian of a lifted smooth function on $M_{2}$.

\begin{lemma}\label{lift}
Let $p \colon M_{2} \to M_{1}$ be a Riemannian submersion. Consider $f \in C^{\infty}(M_{1})$ and its lift $\tilde{f}$ on $M_{2}$. Then we have that:
\begin{enumerate}[topsep=0pt,itemsep=-1pt,partopsep=1ex,parsep=0.5ex,leftmargin=*, label=(\roman*), align=left, labelsep=0em]
\item $\grad \tilde{f} = \widetilde{\grad f}$,
\item $\Delta \tilde{f} = \widetilde{\Delta f} + \langle \widetilde{\grad f} , H \rangle$.
\end{enumerate}
\end{lemma}

\begin{proof}
Both statements follow from straightforward computations, which may be found for instance in \cite[Subsection 2.2]{MR2891739}.\qed
\end{proof}\medskip

Recall that the fibers of a Riemannian submersion are submanifolds of the total space.
This allows us to consider the spectrum of a fiber, with respect to the Riemannian metric inherited by the ambient space. In particular, we regard the bottom of the spectrum of the fiber as a function on the base space. According to the next lemma, this function is upper semi-continuous, while Example \ref{discontinuous} demonstrates that it does not have to be continuous, even if the fibers are minimal.

\begin{lemma}\label{measurable}
Let $p \colon M_{2} \to M_{1}$ be a Riemannian submersion. Then the function $\lambda_{0}(F_{x})$ is upper semi-continuous with respect to $x \in M_{1}$.
\end{lemma}

\begin{proof}
Let $C>0$ and $x \in M_{1}$ such that $\lambda_{0}(F_{x}) < C$. We know from Proposition \ref{bottom} that there exists $f \in C^{\infty}_{c}(F_{x}) \smallsetminus \{0\}$ such that $\mathcal{R}(f) < C$. Observe that $f$ can be extended to an $f \in C^{\infty}_{c}(M_{2})$, and there exists an open neighborhood $U$ of $x$ such that $f$ is non-zero on $F_{y}$ for any $y \in U$. Then $\mathcal{R}(f|_{F_{y}})$ depends continuously on $y \in U$, which shows that there exists an open neighborhood $U^{\prime}$ of $x$ such that $\mathcal{R}(f|_{F_{y}}) < C$ for any $y \in U^{\prime}$. Applying Proposition \ref{bottom} to the Riemannian manifold $F_{y}$ gives that $\lambda_{0}(F_{y}) < C$ for any $y \in U^{\prime}$, as we wished. \qed
\end{proof}

\section{Submersions with fibers of bounded mean curvature}\label{Section Lower bound}

The aim of this section is to prove Theorem \ref{lower bound}. Let $p \colon M_{2} \to M_{1}$ be a Riemannian submersion of possibly non-complete Riemannian manifolds. As in \cite{MR2891739, Bordoni}, for $f \in C^{\infty}_{c}(M_{2})$, its \textit{average} $f_{\av}$ on $M_{1}$ is defined by
\[
f_{\av}(x) := \int_{F_{x}} f.
\]
It is worth to mention that in the published version of \cite{MR2891739} there is a typo in this definition, which was meant to be as above, and is in this way in the arXiv version of \cite{MR2891739}.
Using the first variational formula (similarly to \cite[Lemma 2.2]{MR2891739} and \cite[Formula (1.2)]{Bordoni}), we have that $f_{\av} \in C^{\infty}_{c}(M_{1})$ and its gradient is related to the gradient of $f$ by
\begin{equation}\label{average}
\langle \grad f_{\av}(x) , X \rangle = \int_{F_{x}} \langle \grad f - f H, \tilde{X} \rangle
\end{equation}
for any $x \in M_{1}$ and $X \in T_{x}M_{1}$, where $\tilde{X}$ is the horizontal lift of $X$ on $F_{x}$. The \textit{pushdown} of $f$ on $M_{1}$ is given by
\[
g(x) := \sqrt{(f^{2})_{\av}(x)} = \left( \int_{F_{x}} f^{2} \right)^{1/2}.
\]
This quantity was used by Bordoni to establish spectral estimates for submersions with minimal fibers, and $M_{2}$ closed (cf. \cite[Section 3]{MR2963622} and the references provided there). In the context of Riemannian coverings, a similar quantity was introduced in \cite{BMP1} to derive a spectral estimate, and was used further in \cite{Mine 2} to study coverings preserving the bottom of the spectrum.

\begin{lemma}\label{pushdown estimate}
Consider $f \in C^{\infty}_{c}(M_{2})$ and its pushdown $g$ on $M_{1}$. Then for any $x\in M_{1}$ with $g(x)>0$, the gradient of $g$ satisfies
\[
\| \grad g (x) \|^{2} \leq \int_{F_{x}} \big\| (\grad f)^{h} - f \frac{H}{2} \big\|^{2}.
\]
In particular, $g$ is Lipschitz, and its gradient vanishes at almost any point where $g$ is zero.
\end{lemma}

\begin{proof}
Given $x \in M_{1}$ such that $g(x) > 0$, it is evident that $g$ is differentiable at $x$. Consider an orthonormal basis $\{e_{i}\}_{i=1}^{m}$ of $T_{x}M_{1}$, and denote by $\tilde{e}_{i}$ the horizontal lift of $e_{i}$ on $F_{x}$, $1 \leq i \leq m$. Using (\ref{average}), we obtain that
\begin{eqnarray}
\langle \grad g (x) , e_{i} \rangle^{2} &=& \frac{1}{4g(x)^{2}} \left(\int_{F_{x}} \langle \grad f^{2} - f^{2}H, \tilde{e}_{i} \rangle\right)^{2} \nonumber \\
&=& \frac{1}{g^{2}(x)} \left(\int_{F_{x}} f \big\langle \grad f - f\frac{H}{2}, \tilde{e}_{i} \big\rangle\right)^{2} \nonumber \\
&\leq& \frac{1}{g^{2}(x)} \left(\int_{F_{x}} f^{2} \right) \left(\int_{F_{x}}\big\langle \grad f - f\frac{H}{2}, \tilde{e}_{i} \big\rangle^{2}\right) \nonumber \\
&=&  \int_{F_{x}}\big\langle \grad f - f\frac{H}{2}, \tilde{e}_{i} \big\rangle^{2}, \nonumber
\end{eqnarray}
which proves the asserted inequality, because $\{\tilde{e}_{i}\}_{i=1}^{m}$ spans the horizontal space at each point of $F_{x}$.
Bearing in mind that $g$ is continuous, and on the set where $g$ is positive we have that $g$ is differentiable with bounded gradient, it is easy to see that $g$ is Lipschitz. The proof is completed by Rademacher's Theorem and the fact that $g$ is non-negative. \qed
\end{proof}

\begin{proposition}\label{pushdown}
Let $p \colon M_{2} \to M_{1}$ be a Riemannian submersion, such that the mean curvature of the fibers satisfies $\| H \| \leq C \leq 2 \sqrt{\lambda_{0}(M_{1})}$. Consider $f \in C^{\infty}_{c}(M_{2})$, with $\| f \|_{L^{2}(M_{2})} = 1$, and its pushdown $g \in \Lip_{c}(M_{1})$. Then the Rayleigh quotients of $f$ and $g$ are related by
\[
\mathcal{R}(f) \geq (\sqrt{R(g)} - C/2)^{2} + \int_{M_{1}} \lambda_{0}(F_{x}) g^{2}(x) dx.
\]
\end{proposition}

\begin{proof}
The assumption that $\| f \|_{L^{2}(M_{2})} = 1$ yields that $\| g \|_{L^{2}(M_{1})} = 1$. Lemma \ref{pushdown estimate}, together with the fact that $\| H \| \leq C$ and $\| f \|_{L^{2}(M_{2})} =1$, gives the estimate
\begin{eqnarray}
\mathcal{R}(g) &\leq& \int_{M_{2}} \big\| (\grad f)^{h} - f \frac{H}{2} \big\|^{2} \nonumber \\
&\leq& \int_{M_{2}} \| (\grad f)^{h} \|^{2}  + C \int_{M_{2}} |f| \| (\grad f)^{h}\|  + \frac{C^{2}}{4}\nonumber \\
&\leq& \int_{M_{2}} \| (\grad f)^{h} \|^{2}  + C \left(\int_{M_{2}} \| (\grad f)^{h}\|^{2} \right)^{1/2} + \frac{C^{2}}{4}. \nonumber
\end{eqnarray}
In view of Proposition \ref{bottom}, we have that  $C/2 \leq \sqrt{\lambda_{0}(M_{1})} \leq \sqrt{\mathcal{R}(g)} $, which shows that
\begin{equation}\label{eq1}
\int_{M_{2}} \| (\grad f)^{h} \|^{2} \geq (\sqrt{\mathcal{R}(g)} - C/2)^{2}.
\end{equation}

Recall that at any point of $M_{2}$, the tangent space of $M_{2}$ splits as the orthogonal sum of the horizontal and the vertical space. 
Since $\| f \|_{L^{2}(M_{2})} = 1$, we deduce that
\begin{eqnarray}\label{eq2}
\int_{M_{2}} \| (\grad f)^{h} \|^{2} &=& \int_{M_{2}} \| \grad f \|^{2} - \int_{M_{2}} \| (\grad f)^{v} \|^{2} \nonumber\\
&=& \mathcal{R}(f) - \int_{M_{1}} \int_{F_{x}} \| \grad (f|_{F_{x}}) \|^{2} dx \nonumber \\
&\leq& \mathcal{R}(f) - \int_{M_{1}} \lambda_{0}(F_{x}) \int_{F_{x}} f^{2} dx \nonumber \\
&=& \mathcal{R}(f) - \int_{M_{1}} \lambda_{0}(F_{x}) g^{2}(x) dx,
\end{eqnarray}
where we applied Proposition \ref{bottom} to the fibers. The conclusion is now a consequence of (\ref{eq1}) and (\ref{eq2}). \qed
\end{proof}\medskip

\noindent\emph{Proof of Theorem \ref{lower bound}:} Given $x \in M_{1}$, set $\Lambda(x) := \lambda_{0}(F_{x})$. From Proposition \ref{bottom}, we obtain that for any $\varepsilon > 0$, there exists $f \in C^{\infty}_{c}(M_{2})$, with $\| f \|_{L^{2}(M_{2})} = 1$, such that $\mathcal{R}(f) < \lambda_{0}(M_{2}) + \varepsilon$. Let $g \in \Lip_{c}(M_{1})$ be the pushdown of $f$. Propositions \ref{bottom}, \ref{pushdown}, and the fact that $\| g \|_{L^{2}(M_{1})} = 1$, imply that
\[
\lambda_{0}(M_{2}) + \varepsilon > (\sqrt{\mathcal{R}(g)} - C/2)^{2} + \int_{M_{1}} \Lambda g^{2} \geq (\sqrt{\lambda_{0}(M_{1})} - C/2)^{2} + \inf_{M_{1}} \Lambda,
\]
which proves the first assertion, because $\varepsilon > 0$ is arbitrary.

Suppose now that the equality holds and that $\lambda_{0}(M_{1}) \notin \sigma_{\ess}(M_{1})$. By virtue of Proposition \ref{bottom}, we readily see that there exists $(f_{n})_{n \in \mathbb{N}} \subset C^{\infty}_{c}(M_{2})$, with $\| f_{n} \|_{L^{2}(M_{2})} = 1$, such that 
\[
\mathcal{R}(f_{n}) \rightarrow \lambda_{0}(M_{2}) = (\sqrt{\lambda_{0}(M_{1})} - C/2)^{2} + \inf_{M_{1}} \Lambda.
\]
Consider the sequence $(g_{n})_{n \in \mathbb{N}} \subset \Lip_{c}(M_{1})$ consisting of the pushdowns of $f_{n}$. It follows from Proposition \ref{pushdown} that for any $\varepsilon > 0$ there exists $n_{0} \in \mathbb{N}$ such that
\begin{equation}\label{eq}
\varepsilon > (\sqrt{\mathcal{R}(g_{n})} - C/2)^{2} - (\sqrt{\lambda_{0}(M_{1})} -C/2)^{2} +\int_{M_{1}} (\Lambda - \inf_{M_{1}} \Lambda) g_{n}^{2} 
\end{equation}
for any $n \geq n_{0}$. Notice that the last term is non-negative, and Proposition \ref{bottom} gives the estimate $\sqrt{\mathcal{R}(g_{n})} \geq \sqrt{\lambda_{0}(M_{1})} \geq C/2 $. Thus, (\ref{eq}) yields that $\mathcal{R}(g_{n}) \rightarrow \lambda_{0}(M_{1})$. According to \cite[Propositions 3.5 and 3.7]{Mine2}, after passing to a subsequence, if necessary, we may assume that $g_{n} \rightarrow \varphi$ in $L^{2}(M_{1})$ for some positive $\varphi \in C^{\infty}(M_{1})$ with $\Delta \varphi = \lambda_{0}(M_{1}) \varphi$.

By Lemma \ref{measurable}, we know that for any $c>0$ the set 
\[
A_{c} := \{ x \in M_{1} : \Lambda(x) \geq \inf_{M_{1}} \Lambda + c \}
\]
is closed, and in particular, measurable. For any $n \geq n_{0}$, we conclude from (\ref{eq}) that
\[
\varepsilon \geq \int_{M_{1}} (\Lambda - \inf_{M_{1}} \Lambda) g_{n}^{2} \geq \int_{A_{c}} (\Lambda - \inf_{M_{1}} \Lambda) g_{n}^{2} \geq c \int_{A_{c}} g_{n}^{2} \rightarrow c \int_{A_{c}} \varphi^{2}, \text{ as } n \rightarrow +\infty.
\]
Then $A_{c}$ is of measure zero for any $c>0$, since $\varphi$ is positive in $M_{1}$ and $\varepsilon > 0$ is arbitrary. Hence, $\lambda_{0}(F_{x})$ is equal to its infimum for almost any $x \in M_{1}$. \qed \medskip

The next example shows that in the second statement of Theorem \ref{lower bound}, in general, we do not have that $\lambda_{0}(F_{x}) = 0$ for any $x \in M_{1}$, even if the base manifold is closed and the fibers are minimal.

\begin{example}\label{discontinuous}
Let $(M,g_{0})$ be an $m$-dimensional, non-compact, complete Riemannian manifold with $m \geq 2$ and $\lambda_{0}(M,g_{0}) > 0$. Fix a diverging sequence $(x_{n})_{n \in \mathbb{N}} \subset M$ and $r_{n}>0$ such that the closed balls $C(x_{n}, 3 r_{n})$ are disjoint and the exponential map restricted to the corresponding open ball $\expo \colon B(0 , 3r_{n}) \subset T_{x_{n}}M \to M$ is injective for any $n \in \mathbb{N}$. Consider the compactly supported, Lipschitz functions
\[ 
f_{n}(y) = \left\{
\begin{array}{ll}
1 &  \text{if }d(y,x_{n}) < r_{n}, \\
2 - d(y,x_{n})/r_{n} & \text{if } r_{n}\leq d(y,x_{n}) \leq 2 r_{n},\\
0 &  \text{if } d(y,x_{n}) > 2r_{n}. \\
\end{array} 
\right. 
\]
It is clear that $\grad f_{n}$ vanishes almost everywhere outside $B(x_{n}, 2r_{n}) \smallsetminus C(x_{n},r_{n})$. The restriction of $\grad f_{n}$ in $B(x_{n},2r_{n}) \smallsetminus C(x_{n},r_{n})$ can be extended to a nowhere vanishing, smooth vector field $X_{n}$ in $B(x_{n}, 3r_{n}) \smallsetminus C(x_{n}, r_{n}/2)$ (for instance, $- r_{n}^{-1}\grad d(\cdot , x_{n})$ is such an extension). 
	
For $n \in \mathbb{N}$, consider a positive $\varphi_{n} \in C^{\infty}(M)$, with $\varphi_{n}(y) = 1$ if $d(y,x_{n}) < 3r_{n}/4$ or $d(y,x_{n})>5r_{n}/2$, and $\varphi_{n}(y) = r_{n}^{2}/nc_{n}$ if $r_{n} < d(y,x_{n}) < 2r_{n}$, where
\[
c_{n} := \max \left\{ 1 , \frac{\Vol B(x_{n}, 2r_{n}) }{\Vol B(x_{n} , r_{n})} - 1 \right\}.
\]
Let $\chi_{n} \colon [-1/2,1/2] \to [0,1]$ be an even, smooth function with $\chi_{n}(0) = 0$ and $\chi_{n}(t) = 1$ for $|t| \geq t_{n} := \min\{1 , r_{n}^{2}\}/4nc_{n}$. For each non-zero $-1/2 \leq t \leq 1/2$, define the Riemannian metric $g_{t}$ on $M$, which coincides with the original metric $g_{0}$ outside the union of $B(x_{n} , 3r_{n}) \smallsetminus C(x_{n},r_{n}/2)$ with $n \in \mathbb{N}$, and in any $B(x_{n},3r_{n}) \smallsetminus C(x_{n},r_{n}/2)$ is given by
\[ 
g_{t}(Y,Z) = \left\{
\begin{array}{ll}
(1 - \chi_{n}(t) + \chi_{n}(t) \varphi_{n})^{-1} g_{0}(X_{n},X_{n}) &  \text{if } Y=Z=X_{n}, \\
0 & \text{if } Z=X_{n} \text{ and } g_{0}(Y,X_{n}) = 0,\\
(1 - \chi_{n}(t) + \chi_{n}(t) \varphi_{n})^{1/(m-1)} g_{0}(Y,Z) &  \text{if } g_{0}(Y,X_{n}) =  g_{0}(Z,X_{n}) = 0. \\
\end{array} 
\right. 
\]
for any tangent vectors $Y, Z$. It is elementary to compute
\begin{equation}\label{gradient}
\| {\grad}_{g_{t}} f_{n} \|_{g_{t}}^{2} = r_{n}^{-2} (1 -\chi_{n}(t) + \chi_{n}(t) r_{n}^{2}/nc_{n}) \text{ in } B(x_{n},2r_{n}) \smallsetminus C(x_{n},r_{n}).
\end{equation}
From the fact that the volume element of $g_{t}$ coincides with the volume element of $g_{0}$, we derive that the Rayleigh quotient of $f_{n}$ with respect to the Laplacian corresponding to $g_{t}$ satisfies
\begin{equation}\label{Rayleigh estimate}
\mathcal{R}_{g_{t}}(f_{n}) = \frac{\int_{B(x_{n}, 2r_{n}) \smallsetminus C(x_{n}, r_{n})} \| {\grad}_{g_{t}} f_{n} \|_{g_{t}}^{2}}{\int_{B(x_{n},2r_{n})} f_{n}^{2}} 
\leq c_{n} r_{n}^{-2} (1 - \chi_{n}(t) + \chi_{n}(t) r_{n}^{2}/ nc_{n}),
\end{equation}
where we used that $f_{n} = 1$ in $B(x_{n},r_{n})$. For $t \neq 0$, there exists $n_{0} \in \mathbb{N}$ such that $\chi_{n}(t) = 1$ for any $n \geq n_{0}$. In view of Proposition \ref{bottom}, taking the limit as $n \rightarrow + \infty$ in (\ref{Rayleigh estimate}) gives that $\lambda_{0}(M , g_{t}) = 0$ for any $t \neq 0$.
	
Let $q \colon \mathbb{R} \to S^{1} = \mathbb{R} / \mathbb{Z}$ be the usual Riemannian covering. Consider the product manifold $M \times S^{1}$ endowed with the Riemannian metric $g(x,y) = g_{t} (x) \times g_{S^{1}} (y)$, for $x \in p^{-1}(q(t))$ for some $-1/2 \leq t \leq 1/2$, and $y \in S^{1}$. Then the projection to the second factor $p \colon M \times S^{1} \to S^{1}$ is a Riemannian submersion. Since the volume element of $g_{t}$ is independent from $t$, it is not hard to see that the fibers of $p$ are minimal.

It remains to show that $\lambda_{0}(M \times S^{1}, g) = 0$. To this end, consider $h_{n} \in \Lip_{c}(M \times S^{1})$ defined by $h_{n}(x,y) := f_{n}(x)$. Similarly to (\ref{Rayleigh estimate}), using that $\chi_{n}$ is even, $0 \leq \chi_{n} \leq 1$, and $\chi_{n}(t) = 1$ for $| t | \geq t_{n}$, we obtain that
\[
\mathcal{R}(h_{n}) \frac{\int_{-1/2}^{1/2} \int_{M}  \| {\grad}_{g_{t}} f_{n} \|_{g_{t}}^{2}  dt}{\int_{-1/2}^{1/2} \int_{M}  f_{n}^{2}  dt} 
\leq 2 c_{n} r_{n}^{-2} \int_{0}^{1/2}(1 - \chi_{n}(t) + \chi_{n}(t) r_{n}^{2}/ nc_{n})dt \rightarrow 0,
\]
as $n \rightarrow + \infty$. We conclude from Proposition \ref{bottom} that $\lambda_{0}(M \times S^{1}, g) = 0$, while we have that $\lambda_{0}(F_{q(0)}) > 0$, and $\lambda_{0}(F_{y}) = 0$ for any $y \in S^{1} \smallsetminus \{q(0)\}$.
\end{example}

We now discuss some straightforward applications of Theorem \ref{lower bound}.

\begin{corollary}
Let $p \colon M_{2} \to M_{1}$ be a Riemannian submersion, where $M_{1}$ is an $m$-dimensional Hadamard manifold of sectional curvature $K \leq - a^{2}$ for some $a > 0$. If the mean curvature of the fibers satisfies $\| H \| \leq C \leq (m-1) a$, then
\[
\lambda_{0}(M_{2}) \geq \frac{((m - 1)a - C)^{2}}{4} + \inf_{x \in M_{1}} \lambda_{0}(F_{x}).
\]
\end{corollary}

\begin{proof}
According to McKean's Theorem \cite{MR0266100}, the bottom of the spectrum of $M_{1}$ is bounded by
\[
\lambda_{0}(M_{1}) \geq \frac{(m-1)^{2}a^{2} }{4}.
\]
The asserted inequality is a consequence of Theorem \ref{lower bound}. \qed
\end{proof}\medskip

It is worth to point out that a similar estimate (without the last term) may be derived from \cite[Theorem 5.1]{MR3787357}. However, Theorem \ref{lower bound} yields a sharper estimate than \cite[Theorem 5.1]{MR3787357} for submersions over negatively curved symmetric spaces.

It is well known that an $m$-dimensional negatively curved symmetric space, after rescaling its metric, is isometric to $\mathbb{K} \mathbb{H}^{m}$, where $\mathbb{K}$ is the algebra of real, complex, quaternionic, or Cayley numbers. In the latter case, we have that $m = 2$. The sectional curvature of $\mathbb{K} \mathbb{H}^{m}$ is bounded by $-4 \leq K \leq -1$, the exponential growth of $\mathbb{K} \mathbb{H}^{m}$ is given by
\[
\mu(\mathbb{K} \mathbb{H}^{m}) = m + d - 2, \text{ where } d:= {\dim}_{\mathbb{R}}\mathbb{K},
\]
and the bottom of the spectrum satisfies $\lambda_{0}(\mathbb{K} \mathbb{H}^{m}) = \mu(\mathbb{K} \mathbb{H}^{m})^{2}/4$.

\begin{corollary}
Let $p \colon M_{2} \to \mathbb{K} \mathbb{H}^{m}$ be a Riemannian submersion with fibers of bounded mean curvature $\| H \| \leq C \leq m + d - 2$. Then
\[
\lambda_{0}(M_{2}) \geq (m + d - 2 - C/2)^{2} + \inf_{x \in M_{1}} \lambda_{0}(F_{x}).
\]
\end{corollary}

\begin{proof}
It follows immediately from Theorem \ref{lower bound}. \qed
\end{proof}\medskip

A wider class of examples where Theorem \ref{lower bound} is applicable consists of submersions over complete, negatively curved, locally symmetric spaces.
Any such space, after rescaling its Riemannian metric, is isometric to a regular quotient $M = \mathbb{K} \mathbb{H}^{m} / \Gamma$, where $\Gamma$ is a discrete group. According to the formulas of Sullivan \cite{MR882827} and Corlette \cite{MR1074486}, the bottom of the spectrum of $M$ is given by
\[ 
\lambda_{0}(M) = \left\{
\begin{array}{ll}
\lambda_{0}(\mathbb{K} \mathbb{H}^{m}) &  \text{if } \mu (\Gamma) \leq \mu(\mathbb{K} \mathbb{H}^{m}) / 2, \\
\mu(\Gamma)(\mu(\mathbb{K} \mathbb{H}^{m}) - \mu(\Gamma)) & \text{if }  \mu(\Gamma) \geq \mu(\mathbb{K} \mathbb{H}^{m}) / 2, \\
\end{array} 
\right. 
\]
where $\mu(\Gamma)$ is the exponential growth of $\Gamma$. Theorem \ref{lower bound} can be applied to submersions over such a manifold $M$ in the apparent way.

\section{Submersions with closed fibers}\label{Section closed fibers}

Throughout this section we consider a Riemannian submersion $p \colon M_{2} \to M_{1}$ of complete manifolds with closed fibers. We denote by $V(x)$ the volume of the fiber $F_{x}$ over $x \in M_{1}$, by $H$ the mean curvature vector field of the fibers, and by $X$ the smooth vector field on $M_{1}$ defined by
\[
X(x) := \frac{1}{V(x)} \int_{F_{x}} p_{*} H
\]
for any $x \in M_{1}$. 

Fix an open, bounded domain $U$ of $M_{1}$, and let $f \in C^{\infty}_{c}(M_{2})$ with $f = 1$ in $F_{y}$ for any $y \in U$. Then $f_{\av} = V$ in $U$, and (\ref{average}) shows that $\grad V = - V X$ in $U$. It follows from the fact that $U$ is arbitrary that $\grad V = - V X$ in $M_{1}$. It is immediate to verify that
\[
-\frac{\Delta \sqrt{V}}{\sqrt{V}} = \frac{1}{4} \| X \|^{2} - \frac{1}{2} \diver X.
\]
Let $S$ be the Schr\"{o}dinger operator on $M_{1}$ defined by
\[
S := \Delta - \frac{\Delta \sqrt{V}}{\sqrt{V}} = \Delta + \frac{1}{4} \| X \|^{2} - \frac{1}{2} \diver X.
\]
Observe that $S \sqrt{V} = 0$, which allows us to consider the renormalization $S_{\sqrt{V}}$ of $S$ with respect to $\sqrt{V}$.\medskip

\noindent{\emph{Proof of Theorem \ref{thm2}:}} Let $f \in C^{\infty}_{c}(M_{1}) \smallsetminus \{0\}$ and $\tilde{f}$ its lift on $M_{2}$. By Lemmas \ref{renormalized} and \ref{lift}, we deduce that
\begin{equation}\label{eq lift}
\mathcal{R}(\tilde{f}) = \frac{\int_{M_{2}} \| \grad \tilde{f} \|^{2}}{\int_{M_{2}} \tilde{f}^{2}} = \frac{\int_{M_{1}} \| \grad f \|^{2} V}{\int_{M_{1}} f^{2} V} = \mathcal{R}_{S_{\sqrt{V}}}(f) = \mathcal{R}_{S}(f \sqrt{V}).
\end{equation}
This, together with Proposition \ref{bottom}, proves that $\lambda_{0}(M_{2}) \leq \lambda_{0}(S)$. We know from Corollary \ref{bottom of essential spectrum} that there exists $(f_{n})_{n \in \mathbb{N}} \subset C^{\infty}_{c}(M_{1}) \smallsetminus \{0\}$, with $\supp f_{n}$ pairwise disjoint, such that $\mathcal{R}_{S}(f_{n}) \rightarrow \lambda_{0}^{\ess}(S)$. Then the lifts $\tilde{g}_{n}$ of $g_{n}:=f_{n}/\sqrt{V}$, also have pairwise disjoint supports. Taking into account Proposition \ref{disjoint} and (\ref{eq lift}), it is easy to see that
\[
\lambda_{0}^{\ess}(M_{2}) \leq \liminf_{n} \mathcal{R}(\tilde{g_{n}}) = \liminf_{n} \mathcal{R}_{S}(f_{n}) = \lambda_{0}^{\ess}(S),
\]
which establishes the first assertion.

Suppose now that the submersion has fibers of basic mean curvature. Consider $\lambda \in \mathbb{R}$, $f \in C^{\infty}_{c}(M_{1}) \smallsetminus \{0\}$, and $\tilde{f}$ its lift on $M_{2}$. Using Lemma \ref{lift}, formula (\ref{renormalized exp}), and that $\grad V = - V p_{*}H$, we compute
\begin{eqnarray}
\| (\Delta  - \lambda)\tilde{f} \|^{2}_{L^{2}(M_{2})} &=& \int_{M_{2}} (\widetilde{\Delta f} + \langle \widetilde{\grad f} , H \rangle - \lambda \tilde{f})^{2} \nonumber \\
&=& \int_{M_{1}} (\Delta f + \langle \grad f , p_{*}H \rangle - \lambda f)^{2}V  \nonumber \\
&=& \int_{M_{1}} \bigg(\Delta f - \frac{2}{\sqrt{V}} \langle \grad f , \grad \sqrt{V} \rangle - \lambda f\bigg)^{2}V  \nonumber \\
&=& \| (S_{\sqrt{V}} - \lambda) f \|_{L^{2}_{\sqrt{V}}(M_{1})}^{2}. \nonumber
\end{eqnarray}
In view of Lemma \ref{renormalized}, this implies that
\begin{equation}\label{eq lift 2}
\frac{\| (\Delta  - \lambda)\tilde{f} \|^{2}_{L^{2}(M_{2})}}{\| \tilde{f} \|^{2}_{L^{2}(M_{2})}} = \frac{\| (S_{\sqrt{V}}  - \lambda)f \|^{2}_{L^{2}_{\sqrt{V}}(M_{1})}}{\| f \|^{2}_{L^{2}_{\sqrt{V}}(M_{1})}} = \frac{\| (S  - \lambda)(f \sqrt{V})\|^{2}_{L^{2}(M_{1})}}{\| f \sqrt{V} \|^{2}_{L^{2}(M_{1})}}.
\end{equation}

From Proposition \ref{spectrum}, we obtain that for any $\lambda \in \sigma_{\ess}(S)$, there exists a characteristic sequence $(f_{n})_{n \in \mathbb{N}} \subset C^{\infty}_{c}(M_{1})$ for $S$ and $\lambda$, with $\supp f_{n}$ pairwise disjoint. Then (\ref{eq lift 2}) yields that the sequence $(\tilde{g}_{n})_{n \in \mathbb{N}} \subset C^{\infty}_{c}(M_{2})$, consisting of the lifts of $g_{n}:=f_{n}/\sqrt{V}$, is a characteristic sequence for $\Delta$ and $\lambda$, with $\supp \tilde{g}_{n}$ pairwise disjoint. We conclude from Proposition \ref{spectrum} that $\lambda \in \sigma_{\ess}(M_{2})$. The proof of $\sigma(S) \subset \sigma(M_{2})$ is similar. \qed 

\begin{corollary}\label{closed and bounded}
Let $p \colon M_{2} \to M_{1}$ be a Riemannian submersion of complete manifolds, with closed fibers of bounded mean curvature $\| H \| \leq C$. Then
\[
\lambda_{0}(M_{2}) \leq (\sqrt{\lambda_{0}(M_{1})} + C/2)^{2} \text{ and } \lambda_{0}^{\ess}(M_{2}) \leq (\sqrt{\lambda_{0}^{\ess}(M_{1})} + C/2)^{2}.
\]
\end{corollary}

\begin{proof}
Let $X$ be the vector field and $S$ the Schr\"{o}dinger operator defined in the beginning of this section. Notice that $\| X \| \leq C$. Given $f \in C^{\infty}_{c}(M_{1})$, with $\| f \|_{L^{2}(M_{1})}=1$, we have that
\begin{eqnarray}\label{eq3}
\mathcal{R}_{S}(f) &=& \mathcal{R}(f) + \frac{1}{4} \int_{M_{1}} \| X \|^{2} f^{2} - \frac{1}{2} \int_{M} f^{2} \diver X \nonumber \\
&\leq& \mathcal{R}(f) + \frac{C^{2}}{4} + \frac{1}{2} \int_{M} | \langle \grad f^{2} , X \rangle | \nonumber \\
&\leq& \mathcal{R}(f) + \frac{C^{2}}{4} + C \int_{M} |f| \| \grad f \|  \nonumber \\
&\leq & (\sqrt{\mathcal{R}(f)} + C/2 )^{2},
\end{eqnarray}
where we used the divergence formula and the Cauchy-Schwarz inequality. By virtue of Proposition \ref{bottom}, estimate (\ref{eq3}) shows that $\lambda_{0}(S) \leq (\sqrt{\lambda_{0}(M_{1})} + C/2)^{2}$. The first statement is now a consequence of Theorem \ref{thm2}.

According to Corollary \ref{bottom of essential spectrum}, there exists $(f_{n})_{n \in \mathbb{N}} \subset C^{\infty}_{c}(M_{1})$, with $\supp f_{n}$ pairwise disjoint, $\| f_{n} \|_{L^{2}(M_{1})} = 1$, and $\mathcal{R}(f_{n}) \rightarrow \lambda_{0}^{\ess}(M_{1})$. Bearing in mind Proposition \ref{disjoint} and (\ref{eq3}), it is straightforward to verify that 
\[
\lambda_{0}^{\ess}(S) \leq \liminf_{n} \mathcal{R}_{S}(f_{n}) \leq (\sqrt{\lambda_{0}^{\ess}(M_{1})} + C/2)^{2}.
\]
The proof is completed by Theorem \ref{thm2}. \qed
\end{proof}\medskip

\noindent\emph{Proof of Corollary \ref{discrete spectrum}:} 
Let $C$ be the supremum of the norm of the mean curvature of the fibers. It follows from Corollary \ref{closed and bounded} that if the spectrum $M_{1}$ is not discrete, then the spectrum of $M_{2}$ is not discrete.

Conversely, suppose that $M_{1}$ has discrete spectrum, and let $(K_{n})_{n \in \mathbb{N}}$ be an exhausting sequence of $M_{1}$ consisting of compact subsets of $M_{1}$. From Proposition \ref{exhausting}, we readily see that $(\lambda_{0}(M_{1} \smallsetminus K_{n}))_{n \in \mathbb{N}}$ is an increasing sequence that diverges. In particular, there exists $n_{0} \in \mathbb{N}$ such that $C \leq 2\sqrt{\lambda_{0}(M_{1} \smallsetminus K_{n})}$ for any $n \geq n_{0}$. Applying Theorem \ref{lower bound} to the restriction of $p \colon M_{2} \smallsetminus p^{-1}(K_{n}) \to M_{1} \smallsetminus K_{n}$ over any connected component of $M_{1} \smallsetminus K_{n}$ gives that
\begin{equation}\label{eq4}
\lambda_{0}(M_{2} \smallsetminus p^{-1}(K_{n})) \geq (\sqrt{\lambda_{0}(M_{1} \smallsetminus K_{n})} - C/2)^{2}
\end{equation}
for any $n \geq n_{0}$. Observe that $(p^{-1}(K_{n}))_{n \in \mathbb{N}}$ is an exhausting sequence of $M_{2}$ consisting of compact subsets of $M_{2}$, because $p$ has closed fibers. In view of Proposition \ref{exhausting}, taking the limit as $n \rightarrow + \infty$ in (\ref{eq4}), we derive that $M_{2}$ has discrete spectrum.\qed 
\medskip

Finally, we present some basic examples where our results can be applied. We assume that the manifolds involved in these examples are complete.

\begin{examples}
\begin{enumerate}[topsep=0pt,itemsep=-1pt,partopsep=1ex,parsep=0.5ex,leftmargin=*, label=(\roman*), align=left, labelsep=0em]
\item The \textit{warped product} $M_{2} = M_{1} \times_{\psi} F$ is the product manifold endowed with the Riemannian metric $g_{N} \times \psi^{2} g_{F}$, where $\psi \in C^{\infty}(M_{1})$ is positive. The projection to the first factor $p \colon M_{2} \to M_{1}$ is a Riemannian submersion with fibers of basic mean curvature
\[
H = - k \grad (\ln \tilde{\psi}),
\]
where $k = \dim(F)$. Suppose that $F$ is closed, and consider the Schr\"{o}dinger operator
\[
S := \Delta - \frac{\Delta \psi^{k/2}}{\psi^{k/2}}
\]
on $M_{1}$. Taking into account Theorem \ref{thm2}, we deduce that $\sigma(S) \subset \sigma(M_{2})$ and $\sigma_{\ess}(S) \subset \sigma_{\ess}(M_{2})$. If, in addition, $\grad (\ln \psi)$ is bounded, then Corollary \ref{discrete spectrum} implies that $\sigma_{\ess}(M_{1}) = \emptyset$ if and only if $\sigma_{\ess}(M_{2}) = \emptyset$. It is worth to point out that surfaces of revolution are warped products of the form $\mathbb{R} \times_{\psi} S^{1}$.

\item A wider class of Riemannian submersions than warped products, consists of \textit{Clairaut submersions}, which were introduced by Bishop motivated by a result of Clairaut on surfaces of revolution. A Riemannian submersion $p \colon M_{2} \to M_{1}$ is called Clairaut submersion if there exists a positive $f \in C^{\infty}(M_{2})$, such that for any geodesic $c$ on $M_{2}$, the function $(f \circ c) \sin \theta$ is constant, where $\theta(t)$ is the angle between $c^{\prime}(t)$ and $(T_{c(t)}M_{2})^{h}$. Bishop showed that a Riemannian submersion $p \colon M_{2} \to M_{1}$ with connected fibers, is a Clairaut submersion if and only if the fibers are totally umbilical with mean curvature
\[
H = - k \grad (\ln \tilde{\psi})
\]
for some positive $\psi \in C^{\infty}(M_{1})$, where $k$ is the dimension of the fiber (cf. for instance \cite[Theorem 1.7]{MR2110043}). It is immediate to obtain statements for Clairaut submersions with closed and connected fibers, analogous to the ones we established for warped products.

\item Let $G$ be a compact and connected Lie group acting smoothly and freely via isometries on a Riemannian manifold $M$, with $\dim (M) > \dim (G)$. Then the projection $p \colon M \to M / G$ is a Riemannian submersion with closed fibers of basic mean curvature.
\end{enumerate}
\end{examples}

\noindent Max Planck Institute for Mathematics \\
Vivatsgasse 7, 53111, Bonn \\
E-mail address: polymerp@mpim-bonn.mpg.de

\end{document}